\documentclass[12pt]{article}
\usepackage{latexsym,amssymb,amsmath}
\usepackage{amsthm, amstext}
\usepackage{array, amsfonts, mathrsfs}
\usepackage{mathrsfs}
\usepackage[dvips]{graphicx,color}

\newtheorem{Theorem}{Theorem}
\newtheorem{Lemma}{Lemma}

\newtheorem{Convention}{Convention}
\newtheorem{Corollary}{Corollary}
\newtheorem{Remark}{Remark}

\date{}
\begin{document}
\author{M.I.Belishev\thanks {Saint-Petersburg Department of the Steklov Mathematical Institute, RAS,
                 belishev@pdmi.ras.ru. Supported by the RFBR
                 grants 17-01-00529-a and 18-01-00269.} and
        A.F.Vakulenko\thanks{Saint-Petersburg Department of the Steklov Mathematical Institute,
                 vak@pdmi.ras.ru. Supported by the RFBR
                 grant 18-01-00269.}}
\title{On algebraic and uniqueness properties of 3d harmonic quaternion fields}
\maketitle

\begin{abstract}
Let $\Omega$ be a smooth compact oriented 3-dimensional Riemannian
manifold with boun\-dary. A quaternion field is a pair
$q=\{\alpha,u\}$ of a function $\alpha$ and a vector field $u$ on
$\Omega$. A field $q$ is {\it harmonic} if $\alpha, u$ are
continuous in $\Omega$ and $\nabla\alpha={\rm rot\,}u,\,{\rm
div\,}u=0$ holds into $\Omega$. The space ${\mathscr Q}(\Omega)$
of harmonic fields is a subspace of the Banach algebra $\mathscr
C\left(\Omega\right)$ of continuous quaternion fields with the
point-wise multiplication $qq'=\{\alpha\alpha'-u\cdot u',\,\alpha
u'+\alpha'u+u\wedge u'\}$. We prove a Stone-Weierstrass type
theorem: the subalgebra $\vee{\mathscr Q}(\Omega)$ generated by
harmonic fields is dense in $\mathscr C\left(\Omega\right)$. Some
results on 2-jets of harmonic functions and the uniqueness sets of
harmonic fields are provided.
\end{abstract}


\bigskip

\setcounter{section}{-1}

\section{Introduction}\label{sec Introduction}
\subsubsection*{Motivation}
There is an approach to inverse problems of mathematical physics
(the so-called BC-method), which was originally based on the
relations between inverse problems and the boundary control theory
\cite{B EMP,B EACM,B UMN 2017}. The BC-method recovers Riemannian
manifolds via spectral and/or dynamical boundary data.  Later on,
its version that makes use of connections with Banach algebras,
was proposed in \cite{B Calderon 2003,B Sobolev Geom Rings,B UCLA
2013}.

The problem of recovering the manifold via its DN-map (the
so-called Impedance Tomography Problem) in dimensions $\geqslant
3$ isn't properly solved yet. However, beginning from the papers
\cite{B CUBO 3d tomogr 2005,BSharaf 2008} it becomes clear that
harmonic quaternion fields may play the key role in the 3d ITP. It
is the reason, which has stimulated the study of their properties
\cite{B Quat 2016,BV_8}.

Here we consider certain of algebraic and uniqueness properties of
the harmonic quaternion fields with hope for their future
application to ITP \cite{B Quat 2016}. In the mean time, our
results may be of certain independent interest for the real
uniform Banach algebras theory \cite{Abel Jarosz,Jarosz,Kulkarni}.

\subsubsection*{Main result}
\noindent$\bullet$\,\,\,Let $\Omega$ be a smooth compact oriented
3-dimensional Riemannian manifold with boun\-dary, $T\Omega_x$ the
tangent space at $x\in\Omega$, $u\cdot v$ and $u\wedge v$ the
inner and vector products in $T\Omega_x$. Elements of the space
$H_x:=\mathbb R\oplus T\Omega_x$ (the pairs $q=\{\alpha,u\}$)
endowed with a multiplication $qq'=\{\alpha\alpha'-u\cdot
u',\,\alpha u'+\alpha'u+u\wedge u'\}$ are said to be the {\it
geometric quaternions}. As an algebra, $H_x$ is isometrically
isomorphic to the quaternion algebra $\mathbb H$.
\smallskip

\noindent$\bullet$\,\,\,A {\it quaternion field} is a pair
$q=\{\alpha,u\}$ of a function $\alpha$ and vector field $u$ on
$\Omega$; in other words, $q$ is an $H_x$-valued function on the
manifold. The space $C(\Omega;H)$ of continuous quaternion fields
endowed with the point-wise linear operations and multiplication,
and the relevant $\rm sup$-norm, is a real uniform Banach algebra
\cite{Abel Jarosz,Jarosz,Kulkarni}.

A field $q=\{\alpha,u\}\in C(\Omega;H)$ is {\it harmonic} if
$\alpha, u$ are continuous in $\Omega$ and $\nabla\alpha={\rm
rot\,}u,\,{\rm div\,}u=0$ holds into $\Omega$. The space
${\mathscr Q}(\Omega)$ of harmonic fields is a subspace of
$C(\Omega,H)$ (but not a subalgebra!).
\smallskip

\noindent{$\bullet$}\,\,\,Let $\mathscr A$ be an algebra. For a
set $A\subset\mathscr A$ by $\vee A$ we denote the minimal
subalgebra that contains $A$. The main result of the paper is a
Stone-Weierstrass type Theorem \ref{T1} which claims that
$\vee\mathscr Q(\Omega)$ is dense in $C(\Omega;H)$.

\subsubsection*{More results and comments}
\noindent{$\bullet$}\,\,\,In the course of proving Theorem
\ref{T1} we show that $\mathscr Q(\Omega)$ (and, hence,
$\vee\mathscr Q(\Omega)$) separates points of $\Omega$. It is
almost evident for $\Omega\subset\mathbb R^3$ \cite{BV_8} but far
from being evident for a 3d-manifold of arbitrary topology. The
separation property is derived from the so-called
$H$-controllability of $\Omega$ from the boundary, which is much
stronger than separability. The $H$-controllability is proved by
the use of the results \cite{MT} on existence of the global Green
function and the Landis type uniqueness theorems for the second
order elliptic equations \cite{Leis}. The key step in proving
Theorem \ref{T1} is to show that $\overline{\vee{\mathscr
Q}(\Omega)}$ contains the algebra of scalar fields
$\left\{\{\alpha,0\}\,|\,\,\alpha\in C^\mathbb R(\Omega)\right\}$.
The latter resembles the trick applied in \cite{Holladay}.
\smallskip

\noindent{$\bullet$}\,\,\,In sec \ref{sec 2 jets controllablity}
we prove that the 2-jets of harmonic functions are point-wise
controllable from the boundary. The proof also makes use of the
elliptic uniqueness theorems. Then this result is applied to show
that harmonic functions determine the Riemannian structure of 3d
manifold. As we hope, it is a step towards the main prospective
goal: application to 3d ITP on Riemannian manifolds.
\smallskip

\noindent{$\bullet$}\,\,\,One more result which is of certain
independent interest is the following uniqueness property of
harmonic quaternion fields (sec \ref{sec Uniqueness properties of
harmonic fields}). If $q\in\mathscr Q(\Omega)$ vanishes on a piece
of a smooth surface then it vanishes in $\Omega$ identically.
\smallskip

\noindent{$\bullet$}\,\,\,Everywhere in the paper we deal with
{\it real} functions, fields, spaces, etc. Everywhere {\it smooth}
means $C^\infty$-smooth.

\subsubsection*{Acknowledgements}
We'd like to thank Dr C.Shonkwiler for helpful remarks and useful
references.

\section{Quaternion fields}\label{sec Quaternion
fields}

\subsubsection*{Quaternions}
\noindent$\bullet$\,\,\,Let $E$ be an oriented 3d euclidean space,
$u\cdot v$ and $u\wedge v$ the scalar (inner) and vector products,
$|u|=\sqrt{u\cdot u}$. Elements $p=\{\alpha,u\}$ of the space
$H:={\mathbb R}\oplus E$ endowed with the norm
$|p|=\sqrt{\alpha^2+|u|^2}$ and a (noncommutative) multiplication
 \begin{equation}\label{Eq multiplication}
pp':=\{\alpha \alpha' - u\cdot u', \,\alpha u' + \alpha' u
+u\wedge u'\}\,,
 \end{equation}
are said to be {\it geometric quaternions}.

The norm obeys $|p^2|=|p|^2$,
\smallskip

\noindent$\bullet$\,\,\,Let $\mathbb H$ be the algebra of
(standard) quaternions. Recall that it is the real algebra
generated by ${\bf 1, i, j, k}$ with the unit $\bf 1$ and
multiplication defined by the table
 \begin{equation*}\label{Eq multipl table}
{\bf i^2=j^2=k^2=-1,\quad ij=k,\,\,jk=i,\,\,ki=j}\,.
 \end{equation*}
\smallskip

\noindent$\bullet$\,\,\, For an orthogonal normalized basis
$\varepsilon=\{e_1, e_2, e_3\}$ in $E$, the correspondence
$e_1\mapsto {\bf i},\, e_2\mapsto {\bf j},\, e_3\mapsto {\bf k}$
determines an isometric isomorphism $\mu_\varepsilon: H\to\mathbb
H$,
 \begin{equation}\label{Eq H to H quat}
\{\alpha, \,a e_1+b e_2+c
e_3\}\overset{\mu_\varepsilon\,\,}\mapsto \alpha{\bf 1}+a{\bf
i}+b{\bf j}+c {\bf k}\,,
 \end{equation}
(we write $H\cong\mathbb H$). Any isometric isomorphism $\mu:
H\to\mathbb H$ is of the form (\ref{Eq H to H quat}) by proper
choice of the basis $\varepsilon$.

\subsubsection*{Vector analysis}
In the sequel, the following assumptions are accepted.
 \begin{Convention}\label{C1}
$\Omega$ is a smooth compact oriented Riemannian 3d-manifold with
the smooth boundary $\partial\Omega$. It is endowed with the
metric tensor $g\in C^{2}$; $d\mu$ and $\star$ are the Riemannian
volume 3-form  and the Hodge operator.
 \end{Convention}
On such a manifold, the intrinsic operations of vector analysis
are well defined on smooth functions and vector fields (sections
of the tangent bundle $T\Omega$). Following \cite{Sch}, Chapter
10, we recall their definitions.
\smallskip

\noindent{$\bullet$}\,\,\,For a vector field $u$, one defines the
{\it conjugate $1$-form} $u_\flat$ by
$u_\flat(v)=g(u,v),\,\,\forall v$. For a $1$-form $f$, the {\it
conjugate field} $f^\flat$ is defined by
$g(f^\flat,u)=f(u),\,\,\forall u$.
\smallskip

\noindent{$\bullet$}\,\,\,A {\it scalar product}:\, $\{
\text{fields}\} \times \{\text{fields}\} \, \overset{\cdot\,\,}\to
\, \{ \text{functions}\}$ is defined point-wise by $u\cdot
v=g(u,v)$. A {\it vector product}: $\{ \text{fields}\} \times
\{\text{fields}\} \, \overset{\wedge\,\,}\to \, \{\text{fields}\}$
is defined point-wise by $g(u \wedge v,w)=d\mu\,(u,v,w),\,\forall
w$.
\smallskip

\noindent{$\bullet$}\,\,\,A {\it gradient}: $\{\text{functions}\}
\overset{\nabla}\to \{\text{fields}\}$ and a {\it divergence}: $\{
\text{fields}\} \overset{\rm div}\to \{\text{functions}\}$ are
defined by $\nabla \alpha=(d\alpha)^\flat$ and ${\rm div}\,u=\star
\,d\!\star\, u_\flat$ respectively, where $d$ is the exterior
derivative.
\smallskip

\noindent{$\bullet$}\,\,\,A {\it rotor}:
$\{\text{fields}\}\overset{\rm rot}\to\{\text{fields}\}$ is
defined by ${{\rm rot\,}}u =(\star\, d\,u_\flat )^\flat$. Recall
the basic identities: ${\rm div}\,{{\rm rot\,}}=0$ and ${\rm
rot\,} \nabla =0$. The equalities
 \begin{equation*}
\nabla\alpha\,=\,{\rm rot\,} u \qquad \text{and}\qquad
d\alpha=\star\,d\,u_\flat
  \end{equation*}
are equivalent.
 \smallskip

\noindent{$\bullet$}\,\,\,The {\it Laplacian} $\Delta :\,
\{\text{functions}\} \to \{\text{functions}\}$ is $\Delta ={\rm
div\,}\nabla$. The {\it vector Laplacian} $\vec\Delta :\,
\{\text{fields}\} \to \{\text{fields}\}$ is $\vec \Delta
=\nabla\,{\rm div\,}-{\rm rot\,}{\rm rot\,}$.
 \begin{Remark}\label{R1}
Under the above accepted assumptions on the smoothness of $\Omega$
and $g$, the (harmonic) functions and fields, which obey
$\Delta\alpha=0$ and $\vec\Delta u=0$ in the relevant weak sense,
do belong to the class $C^2_{\rm loc}$: see, e.g,
\cite{BersJhScht}, Part II, Chapter 1.
 \end{Remark}

\subsubsection*{Fields}

Let $\dot\Omega:=\Omega\setminus\partial\Omega$ be the set of the
inner points, $C(\Omega)$ and $\vec C(\Omega)$ the spaces of
continuous functions and vector fields. Let $H_x:=\mathbb R\oplus
T\Omega_x,\,\,\,x\in\Omega$ be the point-wise geometric quaternion
algebras.
\smallskip

\noindent$\bullet$\,\,\,A quaternion {\it field} is a pair
$p=\{\alpha,u\}$ with the components $\alpha\in C(\Omega)$ and
$u\in\vec C(\Omega)$, the values $p(x)=\{\alpha(x),u(x)\}\in H_x$
being regarded as geometric quaternions.

By $C(\Omega;H)$ we denote the space of continuous quaternion
fields. One can regard them as sections of the bundle
$C(\Omega;H)=\cup_{x\in\Omega}H_x$.
\smallskip

\noindent$\bullet$\,\,\, Elements of the subspace
 $$
{{\mathscr Q}(\Omega)}\,:=\,\left\{p \in
C(\Omega;H)\,\big|\,\,\nabla\alpha={\rm rot\,} u,\,\,{\rm div\,}
u=0\,\,\,\text{in}\,\,\,{\dot\Omega}\right\}
 $$
are said to be {\it harmonic fields}. To be rigorous, here the
conditions on the components of $p$ are understood in the relevant
sense of distributions but imply $\Delta\alpha=0$ and $\vec\Delta
u=0$, so that $\alpha$ and $u$ are automatically smooth enough by
Remark \ref{R1}.

\section{Density theorem}\label{sec Density Theorem}
\subsubsection*{Algebra $C(\Omega;H)$}
The space $C(\Omega; H)$ with the point-wise multiplication
(\ref{Eq multiplication}) and the norm
 $$
\|p\|\,=\,\underset{x\in\Omega}{\rm
sup\,}|p(x)|\,=\,\underset{x\in\Omega}{\rm
sup\,}\sqrt{|\alpha(x)|^2+ |u(x)|^2_{T\Omega_x}}
 $$
satisfying $\|qp\|\leqslant\|q\|\|p\|$, $\|p^2\|=\|p\|^2$ is a
real uniform noncommutative Banach algebra.
\smallskip

\noindent{$\bullet$}\,\,\,The fields $\{\alpha,0\}$ constitute a
subalgebra $C(\Omega;\mathbb R)$ of $C(\Omega;H)$, which is
isometrically isomorphic to the real continuous function algebra
on $\Omega$:
 \begin{equation}\label{Eq  C scal cong C^R(Omega)}
C(\Omega;\mathbb R)\,\cong\,C^\mathbb R(\Omega)\,.
 \end{equation}
We say $\{\alpha,0\}$ to be the scalar fields and often identify
them with functions $\alpha$ via the map
$\alpha\mapsto\{\alpha,0\}$, which embeds $C^\mathbb R(\Omega)$ in
$C(\Omega;H)$.
\smallskip

\noindent{$\bullet$}\,\,\,The harmonic subspace $\mathscr
Q(\Omega)\subset C(\Omega;H)$ is not an algebra since, in general,
$p,q\in\mathscr Q(\Omega)$ does not imply $pq\in\mathscr
Q(\Omega)$. It is easy to see that
 $$
\mathscr Q(\Omega)\cap C(\Omega;\mathbb
R)=\left\{\{c,0\}\,|\,\,c\,\,\text{is\,\,a\,\,constant\,\,function}\right\},
 $$
whereas $\{1,0\}$ is the unit of $C(\Omega;H)$.

\subsubsection*{Main result}
For an algebra $\mathscr A$ and a set $S\subset \mathscr A$ by
$\vee S$ we denote a minimal (sub)algebra in $\mathscr A$, which
contains $S$. Our main results is the following.
 \begin{Theorem}\label{T1}
The algebra $\vee \mathscr Q(\Omega)$ is dense in $C(\Omega;H)$.
 \end{Theorem}
The proof occupies the rest of sec \ref{sec Density Theorem}.

\subsubsection*{Green function}
\noindent{$\bullet$}\,\,\,A well-known fact of Geometry is that
the assumptions of Convention \ref{C1}, in particular, provide the
existence of a compact 3-dimensional $C^\infty$- manifold
$\Omega'\Supset\Omega$ {\it without boundary} endowed with the
tensor $g'\in C^2$ such that $g'|_{\Omega}=g$. This enables one to
apply the results by M.Mitrea and M.Taylor \cite{MT} (existence of
the fundamental solution, Green function, Poisson formula, etc)
which are valid for much weaker smoothness restrictions on $g$ and
$\partial\Omega$. Also, one can apply the results on the
uniqueness of continuation of solutions to the elliptic PDE
\cite{BersJhScht,Leis}.
\smallskip

\noindent{$\bullet$}\,\,\,The following results are mostly taken
from \cite{MT}. Also we use some well-known facts of the elliptic
2-nd order equations theory \cite{Miranda,BersJhScht,Leis}. By
$W^l_p(\Omega)$ we denote the Sobolev space of functions which
possess the (generalized) derivatives of the order $l=1,2,\dots$
belonging to $L_p(\Omega)$ ($p\geqslant 1$). Recall that
$\dot\Omega=\Omega\setminus\partial\Omega$. Also we put
$D:=\left\{(x,y)\in\Omega\times\Omega\,|\,\,x=y\right\}$. The
distance in $\Omega$ is denoted by $r_{xy}$.
\smallskip

For an $h\in L_2(\Omega)$, the Dirichlet problem
 \begin{align*}
& \Delta v = h && \text{in}\,\,\,\dot\Omega\\
& v=0 && \text{on}\,\,\,\partial\Omega
 \end{align*}
has a unique solution $v^h\in W^2_2(\Omega)$ vanishing at the
boundary. The solution is represented in the form
 \begin{equation}\label{Eq v^h via G}
v^h(x)\,=\,\int_{\Omega}G(x,y)\,h(y)\,d\mu(y), \qquad x\in \Omega
 \end{equation}
via the {\it Green function} $G$, which possesses the following
properties.
\smallskip

\noindent{\bf 1.}\,\,\,$G\in C^2_{\rm
loc}\left([\Omega\times\Omega]\setminus D\right)$;\quad
$G(x,y)=G(y,x),\,\,\,(x,y)\not\in D$;
 \begin{equation}\label{Eq G=0 on dOmega}
G(x,\cdot)|_{\partial\Omega}\,=\,0,\qquad x\in\dot\Omega\,.
 \end{equation}
For the closed sets $K,K'\subset\Omega$ provided $K\cap
K'=\emptyset$ the map $y\mapsto G(\cdot,y)$ is continuous from $K$
to $C^2(K')$.
\smallskip

\noindent{\bf 2.}\,\,\,The estimates $$G(x,y)\leqslant
\frac{c}{r_{xy}}\,,\qquad |\nabla_y G(x,y)|\leqslant
\frac{c}{r^2_{xy}}$$ hold and imply $G(x,\cdot)\in W^1_p(\Omega)$
for $x\in\Omega,\,\,1\leqslant p<\frac{3}{2}$\,.
\smallskip

\noindent{\bf 3.}\,\,\,As a distribution of the class ${\mathscr
D}^{\,\prime}(\dot\Omega)$ on the test functions (of the variable
$y$) of the class ${\mathscr D}(\dot\Omega)$, the Green function
satisfies
 \begin{equation}\label{Eq Delta G=delta}
\Delta_y G(x,\cdot)\,=\,\delta_x,
 \end{equation}
where $\delta_x$ is the Dirac measure supported at $x$. Note that
in (\ref{Eq Delta G=delta}), and below in (\ref{Eq Delta
deG=delta'}), (\ref{Eq deG=0 on dOmega}), the variable
$x\in\dot\Omega$ plays the role of parameter.
\smallskip

\noindent{\bf 4.}\,\,\,For $f\in C^\infty(\partial\Omega)$, the
inhomogeneous boundary value problem
 \begin{align}
& \Delta w = 0 && \text{in}\,\,\,\dot\Omega \label{Eq D3}\\
& w=f && \text{on}\,\,\,\partial\Omega \label{Eq D4}
 \end{align}
has a unique classical solution $w=w^f(x)$, which is represented
in the form
 \begin{equation}\label{Eq uf via G}
w^f(x)\,=\,\int_{\partial
\Omega}\partial_{\nu_y}G(x,y)\,f(y)\,d\sigma(y),\qquad
x\in\dot\Omega\,,
 \end{equation}
where $\nu_y$ is the outward unit normal at the boundary,
$d\sigma$ is the boundary surface element. This is a Poisson
formula derived from (\ref{Eq v^h via G}) by integration by parts.
Function $f$ in (\ref{Eq D4}) is said to be a {\it boundary
control}.
\smallskip

\noindent{$\bullet$}\,\,\,Fix a point $x\in\dot\Omega$ and a
vector $e\in T\Omega_x,\,\,|e|=1$. Let $\gamma_e$ be the geodesic
that emanates from $x$ in direction $e$. Define a functional
$\partial^x_e\delta_x\in\mathscr D^{\,\prime}(\dot\Omega)$ by
 $$
\langle\partial^x_e\delta_x,\varphi\rangle :=\lim\limits_{\gamma_e
\ni\,\, x'\to
x}\frac{\varphi(x')-\varphi(x)}{r_{xx'}}=\left\langle\lim\limits_{\gamma_e
\ni\,\, x'\to
x}\frac{\delta_{x'}-\delta_x}{r_{xx'}}\,,\,\varphi\right\rangle=e\cdot\nabla\varphi(x)\,.
 $$
The relevant limit passage in (\ref{Eq Delta G=delta}) determines
a derivative $\partial^x_e G(x,\cdot)\in\mathscr
D^{\,\prime}(\dot\Omega)$ which satisfies
 \begin{equation}\label{Eq Delta deG=delta'}
\Delta_y[\partial^x_e G(x,\cdot)]\,=\,\partial^x_e\delta_x\,.
 \end{equation}
In the mean time, by the properties 1 and 2, $\partial^x_e
G(\cdot,y)$ is a (classical) function belonging to $L_p(\Omega)$
for $1\leqslant p<\frac{3}{2}$. Moreover it is {\it harmonic} (and
hence $C^2$-smooth) in $\Omega\setminus\{x\}$ and satisfies
  \begin{equation}\label{Eq deG=0 on dOmega}
\partial^x_e G(x,\cdot)|_{\partial\Omega}=0\,, \qquad x\in\dot\Omega.
 \end{equation}

\noindent{$\bullet$}\,\,\,The relevant limit passage in the
Poisson formula (\ref{Eq uf via G}) implies
 \begin{equation}\label{Eq e nabla uf via G}
e\cdot\nabla w^f(x)\,=\,\int_{\partial
\Omega}\partial_{\nu_y}\left[\partial^x_e
G(x,y)\right]\,f(y)\,d\sigma(y),\qquad x\in\dot\Omega\,.
 \end{equation}

\subsubsection*{$H$-controllability}
\noindent$\bullet$\,\,\,The following result plays the key role in
the proof of Theorem \ref{T1}. Recall that $H_x=\mathbb R\oplus
T\Omega_x\cong\mathbb H$, and $\Omega$ obeys Convention \ref{C1}.

For a set of points $A=\{a_1,\dots, a_N\}\subset \Omega$ define a
$4N$-dimensional space $H_A:=\oplus\sum_{i=1}^N H_{a_i}$ and a map
$M_A: C^\infty(\partial\Omega)\to H_A$:
 $$
f\,\mapsto \,\oplus\sum_{i=1}^N \{w^f(a_i),\nabla w^f(a_i)\}
 $$
(each summand $\{w^f(a_i),\nabla w^f(a_i)\}$ belongs to the
corresponding $H_{a_i}$). We say $\Omega$ to be {\it
$H$-controllable from boundary} if this map is surjective for any
finite set $A$.
 \begin{Lemma}\label{L H controllability}
The manifold $\Omega$ is $H$-controllable from boundary.
 \end{Lemma}
 \begin{proof}
The opposite means that $H_A\ominus{\rm Ran\,}M_A\not=\{0\}$, i.e.
there is a nonzero element $\oplus\sum_{i=1}^N \{\alpha_i,\beta_i
e_i\}\in H_A\,\,\,(\alpha_i,\beta_i\in\mathbb R,\,\,\, |e_i|=1$)\,
such that
 \begin{equation}\label{Eq1 of Lemma 1}
\sum\limits_{i=1}^N \alpha_i w^f(a_i)+\beta_i\, e_i\cdot\nabla
w^f(a_i)=0
 \end{equation}
holds for all $f\in C^\infty(\partial\Omega)$. Show that such an
assumption leads to contradiction.
\smallskip

\noindent$\bf 1.$\,\,\,Let $A\subset\dot\Omega$, i.e., all $a_i$
are the interior points. A function
 \begin{equation}\label{Eq def Phi}
\Phi(y):=\sum\limits_{i=1}^N \alpha_i
G(a_i,y)+\beta_i\partial^x_{e_i}G(a_i,y)
 \end{equation}
satisfies
 \begin{align}
\label{Eq Phi1} & \Delta \Phi =0 && {\rm in}\,\,\,\Omega\setminus
A\\
\label{Eq Phi2} & \Phi|_{\partial\Omega}\,=\,0
 \end{align}
by (\ref{Eq G=0 on dOmega}), (\ref{Eq Delta G=delta}), (\ref{Eq
Delta deG=delta'}), and (\ref{Eq deG=0 on dOmega}).

The relations (\ref{Eq uf via G}), (\ref{Eq e nabla uf via G}) and
(\ref{Eq1 of Lemma 1}) easily follow to
 $$
\int_{\partial\Omega}\partial_\nu\Phi(y)\,f(y)\,d\sigma(y)\,=\,0
 $$
that implies
 \begin{align}
\label{Eq Phi3}
\partial_\nu\Phi|_{\partial\Omega}\,=\,0
 \end{align}
by arbitrariness of $f$.
\smallskip

\noindent$\bf 2.$\,\,\, So, $\Phi$ is harmonic in $\Omega\setminus
A$ and has the {\it zero} Cauchy data at the boundary: see
(\ref{Eq Phi2}) and (\ref{Eq Phi3}). By the well-known uniqueness
property of solutions to elliptic PDE (see, e.g., \cite{Leis},
sec. 4.3, Remark 4.17), we get $\Phi=0$ in $\Omega\setminus A$,
i.e., almost everywhere in $\Omega$.

Since $G(a_i,\cdot)\in W^1_p(\Omega)$ and
$\partial_{e_i}G(a_i,\cdot)\in L_p(\Omega)$, we have $\Phi\in
L_p(\Omega)$ for some $p\geqslant 1$. Therefore, $\Phi$ is a
summable function equal zero a.e. in $\Omega$. Thus, $\Phi=0$  as
a distribution of the class $\mathscr D^{\,\prime}(\dot\Omega)$.

In the mean time, by (\ref{Eq Delta G=delta}) and (\ref{Eq Delta
deG=delta'}) one has
 $$
\Delta\Phi\,=\,\sum\limits_{i=1}^N
\alpha_i\delta_{a_i}+\beta_i\partial^x_{e_i}\delta_{a_i}\,\not=\,0\,,
 $$
i.e., $\Phi$ is a {\it nonzero} element of $\mathscr
D^{\,\prime}(\dot\Omega)$. We arrive at the contradiction that
proves the Lemma for $A\in\dot\Omega$.
\smallskip

\noindent$\bf 3.$\,\,\,Let $A$ contain the points of
$\partial\Omega$.  The smoothness assumptions on $\Omega$ enable
one to provide $\Omega', g'$ obeying Convention \ref{C1} and such
that $\Omega\Subset\Omega'$ and $g'|_{\Omega}=g$ holds. Then one
has $A\subset\dot\Omega'$ that reduces this case to the previous
one.
 \end{proof}
Note that relations between controllability and uniqueness
theorems (like the one used in the proof) are widely exploited in
control theory for PDE (see, e.g., \cite{B UMN 2017}).
\medskip

\noindent$\bullet$\,\,\,Recall that $w^f$ is a harmonic function
that solves (\ref{Eq D3}), (\ref{Eq D4}). As immediate consequence
of Lemma \ref{L H controllability} we have
 \begin{Corollary}\label{Cor 1 |nabla w|^2}
The algebra $\vee\left\{|\nabla w^f|^2\,|\,\,f\in
C^\infty(\Omega)\right\}$ is dense in $C^\mathbb R(\Omega)$.
 \end{Corollary}
Indeed, by Lemma \ref{L H controllability}, for any $a,b\in\Omega$
there is a smooth $f$ such that $|\nabla w^f(a)|^2\not=|\nabla
w^f(b)|^2$, i.e., the functions $|\nabla w^f(\cdot)|^2$ separate
points of $\Omega$. In the mean time, by the same Lemma, there is
no $x_0\in\Omega$, at which all these functions vanish
simultaneously. Hence, by the classical Stone-Weierstrass Theorem
(see, e.g., \cite{Naimark}), the above mentioned density does
hold.
\smallskip

Note that $\{0,\nabla w^f\}\in\mathscr Q(\Omega)$ and $\{0,\nabla
w^f\}^2=-\{|\nabla w^f(\cdot)|^2,0\}\in\vee\mathscr Q(\Omega)$.
Hence, the algebra $\vee\left\{\{|\nabla w^f|^2,0\}\,|\,\,f\in
C^\infty(\Omega)\right\}$ is a subalgebra in $\vee\mathscr
Q(\Omega)$. By (\ref{Eq  C scal cong C^R(Omega)}), Corollary
\ref{Cor 1 |nabla w|^2} implies that this algebra is dense in
$C(\Omega;\mathbb R)$. As a result, denoting
 $$
\mathscr C\,:=\,\overline{\vee\mathscr Q(\Omega)}
 $$
we arrive at the important relation
 \begin{equation}\label{Eq mscr C supset C(Omega;R)}
\mathscr C\,\supset\,C(\Omega;\mathbb R)\,.
 \end{equation}

\subsubsection*{Strong separation}
We say that a family $\mathscr F \subset C(\Omega;H)$ {\it
strongly separates} points (of $\Omega$) if for any $a,b\in\Omega$
and $h_a\in H_a,\,h_b\in H_b$ there is a $p\in\mathscr F$ such
that $p(a)=h_a$ and $p(b)=h_b$ holds \cite{Jarosz}.
 \begin{Lemma}\label{L separation}
The space $\mathscr Q(\Omega)$ strongly separates points.
 \end{Lemma}
\begin{proof}

\noindent$\bullet$\,\,\,Let $\vec L_2(\Omega)$ be the space of
square-integrable vector fields and $\mathscr H:=\{v\in \vec
L_2(\Omega)\,|\,\,{\rm div\,} v=0,\,\,{\rm rot\,} v=0\}$ its
harmonic subspace. The well-known Hodge-Morrey-Friedrichs
decomposition claims that
 \begin{equation}\label{Eq Hodge Morrey Fried}
\mathscr H\,=\,\mathscr G\oplus\mathscr N\,=\,\mathscr
R\oplus\mathscr D\,,
 \end{equation}
where
 \begin{align*}
& \mathscr G:=\{v\in\mathscr H\,|\,\,v=\nabla\alpha\},\quad
\mathscr N:=\{v\in\mathscr H\,|\,\,v\cdot\nu=0\}\,,\\
& \mathscr R:=\{v\in\mathscr H\,|\,\,v={\rm rot\,} u\},\quad
\mathscr D:=\{v\in\mathscr H\,|\,\,v\wedge\nu=0\}\,.
 \end{align*}
(see, e.g., \cite{Sch}, Corollary 3.5.2). The subspaces $\mathscr
N$ and $\mathscr D$ determined by the boundary conditions are
called the Neumann and Dirichlet spaces respectively. Their {\it
finite} dimensions are equal to the Betti numbers: ${\rm
dim\,}\mathscr N=\beta_1,\,\,{\rm dim\,}\mathscr D=\beta_2$
\cite{Sch}. Note that $\mathscr N\cap\mathscr D=\{0\}$ \cite{B
CUBO 3d tomogr 2005, Sch}. Also note that ${\rm dim\,}\mathscr
G={\rm dim\,}\mathscr R=\infty$.
\smallskip

\noindent$\bullet$\,\,\,As a consequence of (\ref{Eq Hodge Morrey
Fried}), a field $v\in\mathscr H$ is represented in the form
$v=\nabla\alpha={\rm rot\,} u$ if and only if $v\in\mathscr
G\cap\mathscr R$ or, equivalently, $v\bot[\mathscr N\dot+\mathscr
D]$.

If $w=w^f(x)$ solves (\ref{Eq D3}), (\ref{Eq D4}) then for any
$d\in\mathscr D$ one has
 $$
(\nabla w^f,d)=\int_\Omega \nabla w^f\cdot
d\,\,d\mu=\int_{\partial\Omega}f\,d\!\cdot\!\nu\,d\sigma\,.
 $$
In the mean time, since $\nabla w^f\in\mathscr G$, the
representation $\nabla w^f={\rm rot\,}u$ holds if and only if
$\nabla w^f\bot\mathscr D$, which is equivalent to
 \begin{equation}\label{Eq w^f bot D}
\int_{\partial\Omega}f\,d\!\cdot\!\nu\,d\sigma\,=\,0\,,\qquad
d\in\mathscr D.
 \end{equation}
In particular, taking $f=1$ one has $w^f=1$ in $\Omega$ and gets
 \begin{equation}\label{Eq 1 bot D}
\int_{\partial\Omega}\,d\!\cdot\!\nu\,d\sigma\,=\,0\,,\qquad
d\in\mathscr D.
 \end{equation}
\noindent$\bullet$\,\,\,Now, fix two distinct points
$a,b\in\Omega$ and elements $h_a=\{c_a, k_a\}\in H_a,\,h_b=\{c_b,
k_b\}\in H_b$. To prove the Lemma we need to show that there is a
smooth $f$, which provides
 \begin{equation}\label{Eq Values}
w^f(a)=c_a,\,w^f(b)=c_b;\quad\nabla w^f={\rm rot\,}u;\quad
u(a)=h_a,\,u(b)=h_b\,.
 \end{equation}
\noindent{\bf Step 1.}\,\,\,At first assume $a,b\in\dot\Omega$.
Let $P_x(y):=\partial_{\nu_y}G(x,y)$ be the Poisson kernel. By
(\ref{Eq uf via G}) for $f=1$ we have
 \begin{equation}\label{Eq Poisson normalized}
\int_{\partial \Omega}P_x(y)\,d\sigma(y)\,=\,1\,, \qquad
x\in\Omega\,.
 \end{equation}
In accordance with (\ref{Eq uf via G}) and (\ref{Eq w^f bot D}),
to satisfy the relations $w^f(a)=c_a,\,w^f(b)=c_b;\,\nabla
w^f={\rm rot\,}u$ in (\ref{Eq Values}) we need to find $f$
provided
 \begin{align*}
& \int_{\partial
\Omega}P_a(y)\,f(y)\,d\sigma(y)=c_a\,,\quad\int_{\partial
\Omega}P_b(y)\,f(y)\,d\sigma(y)=c_b\,;\\
&
\int_{\partial\Omega}f(y)\,d(y)\!\cdot\!\nu\,d\sigma(y)\,=\,0\,,\qquad
d\in\mathscr D,
 \end{align*}
or, equivalently,
 \begin{align}\label{Eq Values equiv}
(P_a,f)=c_a,\, (P_b,f)=c_b\,,\,\,\,f\bot \,\nu\!\cdot\!\mathscr D
 \end{align}
(the inner products in $L_2(\partial\Omega)$), where
$\nu\cdot\mathscr D:=\{\nu\cdot d\,|\,\,d\in\mathscr D\}$.

Comparing (\ref{Eq 1 bot D}) with (\ref{Eq Poisson normalized}),
we conclude that neither $P_a$ nor $P_b$ belong to
$\nu\cdot\mathscr D$. In the mean time, $P_a\not= P_b$ as elements
of $L_2(\partial\Omega)$. Indeed, otherwise we'd have
$w^f(a)=w^f(b)$ for any $f$ that is impossible by Lemma \ref{L
separation}. Hence, ${\rm span}\{P_a,P_b\}\cap\nu\cdot\mathscr D$
may consist of $\{c(P_a-P_b)\,|\,\,c\in\mathbb R\}$ {\it only}. As
a result, to proof the solvability of the linear system (\ref{Eq
Values equiv}) (with respect to $f$) in the case of $c_a\not=c_b$
we must show that $P_a-P_b\not\in\nu\cdot\mathscr D$.
\smallskip

\noindent{\bf Step 2.}\,\,\,Assume the opposite: there is a
$d\in\mathscr D$ such that $P_a-P_b=d\cdot\nu$, and show that this
assumption leads to a contradiction.

Compare the fields $\nabla[G(a,\cdot)-G(b,\cdot)]$ and $d$. Since
$G(a,\cdot)=G(b,\cdot)=0$ on $\partial\Omega$ both of them are
{\it normal} on the boundary. Hence, by the assumption, they are
{\it equal} on $\partial\Omega$. In the mean time, the field
$\nabla[G(a,\cdot)-G(b,\cdot)]$ is harmonic in
$\dot\Omega\setminus[\{a\}\cup\{b\}]$, whereas $d$ is harmonic in
the whole $\dot\Omega$. The coincidence at the boundary implies
the coincidence in the domain of harmonicity. Hence,
$\nabla[G(a,\cdot)-G(b,\cdot)]$ can be extended by continuity to
the whole $\Omega$ and $\nabla[G(a,\cdot)-G(b,\cdot)]=d$
everywhere. However, the latter is impossible since
 $$
{\rm div\,}
\nabla[G(a,\cdot)-G(b,\cdot)]=\Delta[G(a,\cdot)-G(b,\cdot)]=\delta_a-\delta_b\,,
 $$
whereas ${\rm div\,}d=0$ everywhere in $\dot\Omega$. This
contradiction shows that $P_a-P_b\not\in\nu\cdot\mathscr D$.
\smallskip

\noindent{\bf Step 3.}\,\,\,The case of $a$ and/or $b$ belonging
to the boundary is reduced to the previous one by the collar
theorem arguments, which were applied at the end of the proof of
Lemma \ref{L H controllability}.
\end{proof}

 \begin{Corollary}\label{Cor strong separation}
The algebra $\vee\mathscr Q(\Omega)\subset C(\Omega;H)$ strongly
separates points of $\Omega$.
 \end{Corollary}
This property plays important role in proving density theorems
\cite{Jarosz}.

\subsubsection*{Completing the proof of Theorem \ref{T1}}
Recall that $\mathscr C=\overline{\vee\mathscr Q(\Omega)}$ and
prove that $\mathscr C= C(\Omega;H)$. The fact, which will play
the key role, is the embedding $\mathscr C\supset C(\Omega;\mathbb
R)\cong C^\mathbb R(\Omega)$: see (\ref{Eq mscr C supset
C(Omega;R)}).
\medskip

\noindent$\bullet$\,\,\,Fix an $x\in\Omega$ and choose the smooth
boundary controls $f^x_1, f^x_2, f^x_3$ such that $\nabla
w^{f^x_1}(x), \nabla w^{f^x_2}(x), \nabla w^{f^x_3}(x)$ constitute
a basis of $T\Omega_x$. It is possible owing to Lemma \ref{L H
controllability}. By continuity, there is a ball
$B_{r(x)}[x]\subset \Omega$ centered at $x$, of (small enough)
radius $r(x)$, such that $\nabla w^{f^x_1}(y), \nabla
w^{f^x_2}(y), \nabla w^{f^x_3}(y)$ is a basis of $T\Omega_y$ for
each $y\in B_{r(x)}[x]$.

Let such a choice be done for each $x\in\Omega$.
\smallskip

\noindent$\bullet$\,\,\,The balls provide an open cover
$\Omega=\cup_{x\in\Omega}B_{r(x)}[x]$. By compactness there is a
finite subcover $\Omega=\cup_{n=1}^N B_{r_n}[x_n]$, where
$r_n:=r(x_n)$. Let $\eta_1,\dots,\eta_N$ be a partition of unit
subordinated to the subcover, so that
 $$
\eta_1,\dots,\eta_N \in C^\infty(\Omega),\quad{\rm
supp\,}\eta_n\subset
B_{r_n}[x_n],\quad\sum\limits_{n=1}^N\eta_n\equiv
1\,\,\,\text{in}\,\,\,\Omega
 $$
holds.
\smallskip

\noindent$\bullet$\,\,\,Take $p=\{\alpha,u\}\in C(\Omega;H)$ and
represent
 \begin{align*}
p=\sum\limits_{n=1}^N\eta_n p=\{\sum\limits_{n=1}^N\eta_n
\alpha,\sum\limits_{n=1}^N\eta_n u\}=\sum\limits_{n=1}^N\{\eta_n
\alpha,0\}+ \sum\limits_{n=1}^N\{0,\eta_n u\}
 \end{align*}
with $\{\eta_n \alpha,0\}\in C(\Omega;\mathbb R)\subset\mathscr
C$. In the mean time, one has
 $$
\eta_n u=\sum\limits_{k=1}^3 \varkappa^n_k\,\nabla w^{f^{x_n}_k}
 $$
with the certain $\varkappa^n_k\in C^\mathbb R(\Omega)$ supported
in $B_{r_n}[x_n]$. Note that $\{\varkappa^n_k,0\}\in
C(\Omega;\mathbb R)\subset\mathscr C$.

Resuming, we arrive at the representation
 \begin{align*}
p=\sum\limits_{n=1}^N\{\eta_n
\alpha,0\}+\sum\limits_{n=1}^N\sum\limits_{k=1}^3\{\varkappa^n_k,0\}\{0,\nabla
w^{f^{x_n}_k}\}\,,
 \end{align*}
where all cofactors and summands do belong to $\mathscr C$. Thus
$p\in\mathscr C$ and, hence, $C(\Omega;H)=\mathscr C$.

{\it Theorem 1 is proved}.

 \begin{Remark}
Analyzing the proof, it is easy to recognize that the family
$\mathscr W:=\left\{\{0,\nabla
w^f\}\,|\,\,\,f\,\,\text{is\,\,smooth}\,\right\}$, which is
smaller than $\mathscr Q(\Omega)$, also generates the whole of the
continuous field algebra: $\overline{\vee\mathscr W}=C(\Omega;H)$.
 \end{Remark}

\section{Controllability of 2-jets}\label{sec 2 jets controllablity}
Fix an $a\in \dot \Omega$; let $x^1, x^2, x^3$ be the local
coordinates in a neighborhood $\omega \ni a$. With a smooth
function $\phi$ one associates the row of its 0,1,2-order
derivatives
 \begin{align*}
&
j_a[\phi]\,:=\{\phi(a);\,\,\phi_{x^1}(a),\phi_{x^2}(a),\phi_{x^3}(a);\,\\
& \phi_{x^1 x^1}(a), \phi_{x^1 x^2}(a),\phi_{x^1 x^3}(a),\phi_{x^2
x^2}(a),\phi_{x^2 x^3}(a),\phi_{x^3 x^3}(a)\}\,\in\,\mathbb
R^{10},\quad
 \end{align*}
which provides a coordinate representation of its {\it second jet}
at the point $a$ \cite{Naram}. For short, we say $j_a[\phi]$ to be
a 2-jet of $\phi$ at $a$ and consider $\mathbb R^{10}$ with the
(standard) inner product $\langle j,j'\rangle$ as a space of
2-jets.

Recall that in coordinates the Laplacian acts by
 $$
\Delta
\phi\,=\,g^{-\frac{1}{2}}[g^{\frac{1}{2}}g^{ik}\phi_{x^k}]_{x^i}\,,
 $$
where $\{g^{ik}\}$ is the inverse to the metric tensor matrix
$\{g_{ik}\}$ and $g={\rm det}\{g_{ik}\}$ (sum\-mation over
repeating indexes is in the use). We say the row
 \begin{align*}
&
\lambda_a\,:=\\
&
=\{0;g^{-\frac{1}{2}}[g^{\frac{1}{2}}g^{i1}]_{x^i},g^{-\frac{1}{2}}[g^{\frac{1}{2}}g^{i2}]_{x^i},g^{-\frac{1}{2}}
[g^{\frac{1}{2}}g^{i3}]_{x^i}; g^{11},
2g^{12},2g^{13},g^{22},2g^{23},g^{33}\}\big|_{x=a}
 \end{align*}
to be the {\it Laplace jet} and represent $(\Delta
\phi)(a)=\langle \lambda_a,j_a[\phi]\rangle$.

The harmonicity $\Delta w=0$ is equivalent to the orthogonality
$\langle j_a[w], \lambda_a\rangle=0,\,\,\,a\in\omega$. Therefore
one has $j_a[w]\in \mathbb R^{10}\ominus {\rm span\,}\lambda_a$.
Let us show that the 2-jets of harmonic functions exhaust the
subspace $\mathbb R^{10}\ominus {\rm span\,}\lambda_a$. This
result may be interpreted as a point-wise boundary controllability
of 2-jets by harmonic functions. Recall that $w^f$ is a solution
to (\ref{Eq D3}), (\ref{Eq D4}).

 \begin{Lemma}\label{L jets}
For any $a\in\Omega$ and $s\in \mathbb R^{10}\ominus\,{\rm
span\,}\lambda_a$ there is a smooth $f$ such that $j_a[w^f]=s$.
 \end{Lemma}
\begin{proof}
Taking into account the structure of the Laplace jet, we may deal
with $s=\{0; s_1,s_2,s_3; s_{11}, \dots, s_{33}\}$, and let it be
such that $0\not=s\in \mathbb R^{10}\ominus\,{\rm
span\,}\lambda_a$ but $\langle s, j_a[w^f]\rangle=0$ for any
smooth $f$. Show that such an assumption leads to contradiction.
\smallskip

\noindent$\bullet$\,\,\,For a differential operator $L$ with
smooth coefficients in $\Omega$, by $L^*$ we denote its adjoint by
Lagrange that is defined by
 $$
(L\eta,\zeta)_{L_2(\Omega)}\,=\,(\eta,L^*\zeta)_{L_2(\Omega)},
\qquad \eta,\zeta \in \mathscr D(\dot\Omega)\,.
 $$
For a distribution $h\in\mathscr D^{\,\prime}(\dot\Omega)$ one
defines $Lh$ by
$(Lh,\eta):=(h,L^*\eta)_{L_2(\Omega)},\,\,\eta\in\mathscr
D(\dot\Omega)$.

Let $S$ be a differential operator, which acts by
 \begin{align*}
& (Sv)(x) \,=\\
& =\, \left[s_1v_{x^1}+s_2v_{x^2}+s_3v_{x^3}+s_{11}v_{x^1
x^1}+s_{12}v_{x^1 x^2}+\dots+s_{33}v_{x^3 x^3}\right](x)=\\
& =\langle s,j_x[v]\rangle, \qquad x\in\omega
 \end{align*}
in a coordinate neighborhood $\omega$ of $a\in\dot\Omega$, where
the (constant) coefficients are the components of the above chosen
jet $s$.
\smallskip

\noindent$\bullet$\,\,\,Let $\delta_a\in\mathscr
D^{\,\prime}(\dot\Omega)$ be the Dirac measure supported at the
point $a\in\dot\Omega$. Consider the problem
 \begin{align}
\label{Eq H1} & \Delta H\,=\,S^*\delta_a\\
\label{Eq H2} & H\big|_{\partial\Omega}\,=\,0\,.
 \end{align}
The equation is understood as a relation in $\mathscr
D^{\,\prime}(\dot\Omega)$; its r.h.s. is a distribution acting by
$(S^*\delta_a,\eta)_{L_2(\Omega)} = (S\eta)(a)$. The boundary
condition does make sense since $H$ is harmonic outside ${\rm
supp\,}S^*\delta_a=\{a\}$. Also, the normal derivative
$\partial_\nu H$ is a smooth function on $\partial\Omega$.

Formally by Green, for a function $v\in C^2(\Omega)$ one has
 \begin{align*}
& \langle s,j_a[v]\rangle=(Sv)(a)=\int_\Omega \delta_a
\,Sv\,d\mu=\int_\Omega S^*\delta_a \,v\,d\mu & \overset{(\ref{Eq
H1})}= \int_\Omega \Delta H \,v\,d\mu=\\
& \overset{(\ref{Eq H2})}= \int_\Omega H \,\Delta v\,d\mu +
\int_{\partial\Omega}
\partial_\nu H\,v\,d\sigma\,.
 \end{align*}
To justify the final equality
 \begin{align}\label{Eq Green for jet}
\langle s,j_a[v]\rangle\,=\,\int_\Omega H \,\Delta v\,d\mu +
\int_{\partial\Omega}
\partial_\nu H\,v\,d\sigma
 \end{align}
one can use the standard regularization technique, approximating
$\delta_a$ by $\delta_a^\varepsilon\in\mathscr D(\dot\Omega)$
supported near $a$.
\smallskip

\noindent$\bullet$\,\,\,By the choice of $s$, for $v=w^f$ the
equality (\ref{Eq Green for jet}) provides
 $$
\int_{\partial\Omega}\partial_\nu
H\,w^f\,d\sigma=\int_{\partial\Omega}\partial_\nu
H\,f\,d\sigma=0\,.
 $$
By arbitrariness of $f$ we get $\partial_\nu H=0$ on
$\partial\Omega$. So, $H$ is harmonic in $\Omega\setminus\{a\}$
and has the zero Cauchy data on the boundary. By the uniqueness
theorem, $H$ vanishes everywhere outside $a$. Hence, the
distribution $H$ is supported at $a$. The well-known fact of the
distribution theory is that such an $H$ is a linear combination of
$\delta_a$ and its derivatives. In the mean time, comparing the
orders of singularities in the left and right hand sides of
(\ref{Eq H1}), one easily concludes that $$H=c\delta_a$$ with
$c={\rm const}\not=0$. Indeed, otherwise $\Delta H$ contains the
derivatives of $\delta_a$ of the order $\geqslant 3$ that makes
the equality (\ref{Eq H1}) impossible.

For an $\eta\in \mathscr D(\dot\Omega)$ one has
 \begin{align*}
 \langle s,j_a[\eta]\rangle= (\delta_a, S\eta) = (S^*\delta_a,
 \eta)\overset{(\ref{Eq H1})}=(\Delta c\delta_a,\eta)=(c\delta_a,\Delta
 \eta)= \langle c\lambda_a,j_a[\eta]\rangle\,.
 \end{align*}
Comparing the beginning with the end and referring to the evident
$\{j_a[\eta]\,|\,\,\eta\in \mathscr D(\dot\Omega)\}=\mathbb
R^{10}$, we arrive at $s=c\lambda_a$ that contradicts to the
starting assumption $s\bot\lambda_a$.
\smallskip

\noindent$\bullet$\,\,\,The case $a\in\partial\Omega$ is reduced
to the previous one by means of the trick already used at the end
of the proof of Lemma \ref{L H controllability}: embedding
$\Omega\Subset\Omega^\prime$.
\end{proof}
\medskip

As is easy to recognize, Lemma \ref{L jets} implies the assertion
of Lemma \ref{L H controllability} for the case of the single
point $a$. However, Lemma \ref{L jets} may be generalized on the
finite set $a_1,\dots, a_N$ so that the relevant boundary
controllability of 2-jets of harmonic functions holds up to the
natural defect in $\oplus\sum_i\mathbb R^{10}_{a_i}$.

\subsubsection*{Determination of metric from harmonic functions}
The metric on $\Omega$ determines the family of harmonic
functions. The converse is also true in the following sense.
\smallskip

\noindent$\bullet$\,\,\,Let $c>0$ be a smooth function on $\Omega$
and $cg$ a conformal deformation of the metric $g$. By
$\Delta_{cg}$ and $\Delta_{g}$ we denote the corresponding
Laplacians. A simple calculation leads to the relation
 \begin{equation}\label{Eq conformal deformation}
\Delta_{cg}y\, =\,c^{-1}\Delta_g y -2^{-1}\,\nabla
c^{-1}\cdot\nabla y\,,
 \end{equation}
which is specific for the 3d case. Taking $y=w^f$, we see that the
metrics $cg$ and $g$ have the same reserve of harmonic functions
$w^f$ if and only if $\nabla c^{-1}\cdot\nabla w^f=0$ holds for
any smooth $f$. In the mean time, by Lemma \ref{L H
controllability} the gradients $\nabla w^f=0$ constitute the local
bases in $\Omega$. Hence, the latter equality implies $\nabla
c^{-1}=0$, i.e., $c=\rm const$.
\smallskip

\noindent$\bullet$\,\,\,Fix a point $a$ in a coordinate
neighborhood $\omega\ni a$. By $\lambda_a^g$ we denote the Laplace
jet of the given metric $g$. By Lemma \ref{L jets}, the space of
jets is
 \begin{equation}\label{Eq decomposition of jets}
\mathbb
R^{10}_a=\{j_a[\phi]\,|\,\,\phi\,\,\text{is\,\,smooth}\}=\{j_a[w^f]\,|\,\,f\,\,\text{is\,\,smooth}\}\oplus
{\rm span\,}\lambda^g_a\,.
 \end{equation}
Therefore, writing  $(\Delta w^f)(a)=0$ in the form
 $$
\langle \lambda^g_a,j_a[w^f]\rangle=0,\qquad
f\,\,\text{is\,\,smooth}
 $$
and varying $f=f_1, f_2, \dots$, we get a linear homogeneous
algebraic system with respect to the components of the jet
$\lambda^g_a$, which determines them up to a factor, which may
depend on $a$. Along with the components, we determine the tensor
$g$ up to a factor, possibly depending on $a$. However, by the
above mentioned geometric reasons, this factor is a constant.
\smallskip

Thus, the family $\{w^f\,|\,\,f \,\,\text{is\,\,smooth}\}$
determines the metric $g$ up to a constant positive factor. If $g$
is known at least at a single point $x_0\in\Omega$, then it is
uniquely determined everywhere.
\smallskip

Notice in addition that in two-dimensional case relation (\ref{Eq
conformal deformation}) is of the form $\Delta_{cg}y\,
=\,c^{-1}\Delta_g y$, so that the metrics $cg$ and $g$ determine
the same reserve of harmonic functions. It is the reason, because
of which in 2d impedance tomography problem the metric is
recovered up to conformal equivalence \cite{B Calderon 2003}.
\medskip

\noindent$\bullet$\,\,\,Here we describe a trick, which is used in
dynamical/spectral inverse problems and 2d impedance tomography
problem, for recovering the metric via boundary data\cite{B UMN
2017}. The hope is that it may be useful in future investigation
of 3d ITP.
\smallskip

Assume that a topological space $\tilde\Omega$ is {\it
homeomorphic} to $\Omega$ via a homeomorphism
$\beta:\Omega\to\tilde\Omega$. Also assume that the family of
functions
 $$
\{\tilde w^f=w^f\circ\beta^{-1}\,|\,\,f \,\,\text{is\,\,smooth}\}
 $$
is given. The following procedure enables one to determine the
metric $\tilde g=\beta_* g$ in $\tilde\Omega$.
\smallskip

\noindent$1.$\,\,\,Fix a point $a\in\tilde\Omega$ and choose its
neighborhood $\tilde\omega$ with the coordinates $x^1, x^2, x^3$.
By the way, Lemma \ref{L H controllability} enables one to use the
images $\tilde w^f$ as local coordinates.
\smallskip

\noindent$2.$\,\,\,Find ${\rm span\,}\lambda_a^{\tilde g}$ by
(\ref{Eq decomposition of jets}) (replacing functions $w^f$ on
$\omega$ with $\tilde w^f$ on $\tilde\omega$). As was shown above,
the family of these subspaces given for $a\in\tilde\omega$
determines the metric up to a constant factor. So, $c\tilde g$ is
recovered. Assuming $\tilde g$ to be known at least at a single
point $a_0\in\tilde\omega$, one recovers $\tilde g$ uniquely.
\smallskip

\noindent$3.$\,\,\, Covering $\tilde\Omega$ by the coordinate
neighborhoods and repeating the previous steps, we determine
$\tilde g$ in $\tilde\Omega$.

\section{Uniqueness properties of harmonic fields}\label{sec Uniqueness properties of harmonic fields}
Roughly speaking, the following result means that the set of zeros
of a harmonic quaternion field may be at most of dimension 1.
 \begin{Lemma}\label{L uniqueness}
Let $\Sigma\in\Omega$ be a $C^2$-smooth surface (2-dim
submanifold). If $p\in\mathscr Q(\Omega)$ obeys $p|_{\Sigma}=0$
then $p=0$ in the whole $\Omega$.
 \end{Lemma}
\begin{proof}
Since the claimed result is of local character, we assume $\Sigma$
to be a both-side surface endowed with a smooth field of the unit
normals $\nu$. Also, $\Sigma$ possesses the (induced) Riemannian
metric and is provided with the corresponding operations on vector
fields. In particular, a divergence, which is denoted by ${\rm
div}_\Sigma$, is well defined.
\smallskip

\noindent$\bullet$\,\,\, For a point $x\in\Sigma$ and vector $v\in
T\Omega_x$ we represent
 $$
v=v_\theta+v_\nu:\qquad v_\nu= v\cdot\nu\,\nu,\quad
v_\theta=v-v_\nu\,
 $$
and, by default, identify $v_\theta$ with the proper vector of
$T\Sigma_x$. By the latter, for a smooth {vector field} $v$ given
in a neighborhood of $\Sigma$, the value $[{\rm div}_\Sigma\,
v_\theta](x)$ is of clear meaning. Also, recall the well-known
vectot analysis relation
 \begin{equation}\label{Eq Lemma uniq 1}
\nu\cdot{\rm rot\,}v\,=\,{\rm div}_\Sigma\, \nu\wedge v_\theta
\qquad {\rm on}\,\,\,\,\Sigma
\end{equation}
(see, e.g. \cite{Sch}).
\smallskip

\noindent$\bullet$\,\,\,Begin with the case $\Sigma\subset
\dot\Omega$. Let $p=\{\alpha,u\}\in \mathscr Q(\Omega)$, so that
 \begin{equation}\label{Eq Lemma uniq 2}
\nabla\alpha\,=\,{\rm rot\,}u,\quad{\rm div\,}u\,=\,0 \qquad {\rm
in}\,\,\,\dot\Omega
\end{equation}
\smallskip
holds. Let $p|_\Sigma=0$. Since $\alpha|_\Sigma =0$, we have
$(\nabla\alpha)_\theta|_\Sigma=0$ that implies $({\rm
rot\,}u)_\theta\big|_\Sigma=0$ by (\ref{Eq Lemma uniq 2}). In the
mean time, $u|_\Sigma=0$ is equivalent to $u_\theta = u_\nu=0$ on
$\Sigma$; hence $({\rm rot\,}u)_\nu|_\Sigma={\rm
div}_\Sigma\,\nu\wedge u_\theta=0$ by virtue of (\ref{Eq Lemma
uniq 1}). Thus we get $({\rm rot\,}u)_\theta|_\Sigma=({\rm
rot\,}u)_\nu|_\Sigma=0$, i.e. ${\rm rot\,}u|_\Sigma=0$.

The latter equality and (\ref{Eq Lemma uniq 2}) lead to
$(\nabla\alpha)|_\Sigma=0$ (along with $\alpha|_\Sigma=0$). So,
$\alpha$ is a harmonic function with the zero Cauchy data on
$\Sigma$. Therefore $\alpha =0$ in $\Omega$ by the elliptic
uniqueness theorems \cite{Leis}.

As a result, ${\rm rot\,}u=\nabla\alpha=0$ everywhere in $\Omega$.
Since ${\rm div\,} u=0$, the vector field $u$ is {\rm harmonic} in
$\Omega$ and vanishes on $\Sigma$. Therefore, locally near the
points $x\in\Sigma$ one represents $u=\nabla\varphi$ with a {\it
harmonic} function $\varphi$ provided $\nabla\varphi|_\Sigma=0$.
Such a function is a constant; hence $u=0$ near $\Sigma$. By its
harmonicity, $u$ vanishes globally in $\Omega$.

So, we have $p=0$ in $\Omega$.
\smallskip

\noindent$\bullet$\,\,\,The case $\Sigma\subset\partial\Omega$ is
reduced to the previous one by means of the trick already used at
the end of the proof of Lemma \ref{L H controllability}: embedding
$\Omega\Subset\Omega^\prime$.
\end{proof}
\medskip

\subsubsection*{The authors}
{\bf Mikhail I. Belishev}. Saint-Petersburg Department of the
Steklov Mathematical Institute, Russian Academy of Sciences;
Saint-Petersburg State University; belishev@pdmi.ras.ru.
\smallskip

\noindent{\bf Aleksei F. Vakulenko}. Saint-Petersburg Department
of the Steklov Mathematical Institute, Russian Academy of
Sciences; vak@pdmi.ras.ru.

\end{document}